\documentclass[12pt]{article}
\usepackage[utf8]{inputenc}
\usepackage{amsmath}
\usepackage{amsthm}
\usepackage{amssymb}
\usepackage[mathscr]{euscript}
\usepackage{enumerate}
\usepackage{graphicx}
\usepackage{fullpage}
\usepackage{tikz}
\usepackage{color}
\usepackage{multicol}
\usepackage{hyperref}

\newtheorem{thm}{Theorem}[section]
\newtheorem{lem}[thm]{Lemma}
\newtheorem{prop}[thm]{Proposition}
\newtheorem{cor}[thm]{Corollary}

\theoremstyle{definition}

\newcommand*\C{\mathcal{C}}
\newcommand*\I{\mathcal{I}}
\newcommand*\nestedset{\mathcal{N}}
\newcommand*\shiftedset{\mathcal{S}}
\newcommand*\KG{\operatorname{KG}}

\newcommand*\sym{\operatorname{Sym}}
\newcommand*\aut{\operatorname{Aut}}
\newcommand*\lk{\operatorname{lk}}
\newcommand*\st{\operatorname{st}}
\newcommand*\fc{\chi^*}
\newcommand{\harm}[1]{H_{\lfloor #1 \rfloor}}
\newcommand{\strata}{\mathfrak{s}}

\title{Endomorphisms of Total Kneser Graphs}
\author{Esmaeil Alhaj Mostafa Abdan, Jacob Russell, Sophia Ulvestad}

\begin{document}

\maketitle
\begin{abstract}
    Inspired by work of Gaster, Greene, and Vlamis on the chromatic number of curve graphs, we prove that every endomorphism of the total Kneser graph is an automorphism. We also show that the automorphism group of the total Kneser graph is the product of a pair of permutation groups.
\end{abstract}

\section{Introduction}
Gaster, Greene, and Vlamis introduced the \emph{total Kneser graph} to study the chromatic number of the curve graph of surfaces \cite{GGV_Coloring_curves}. The total Kneser graph, $\KG(n)$, has a vertex for every partition of $\{1,\dots, n\}$ into two non-empty sets $\{A, A^c\}$ with an edge between $\{A,A^c\}$ and $\{B,B^c\}$ if either $A$ and $B$ are disjoint or one nests into the other.
The total Kneser graph  has a \emph{stratification} by  classical Kneser graphs; the subgraph $\KG(n,k)$ spanned by vertices where one of the sets in the partition has size $k$ is isomorphic to the Kneser graph when $n > 2k$. 

The classical Kneser graph is usually defined as the graph with a  vertex for each subset of $\{1,\dots,n\}$ of size $k$ with edges for disjointness. The Kneser graph is a central object in combinatorics, especially in the study of chromatic numbers (and their variants) of graphs; see \cite[Chapter 7]{GR_Alg_Graph_Theory_Book}. A basic fact about Kneser graphs is that they are \emph{cores}, that is,  the only  endomorphisms are automorphisms. Gaster, Greene, and Vlamis speculated that the total Kneser graph should also be a core when $n \geq 5$.
Our main theorem verifies this conjecture.

\begin{thm}\label{into:core}
    $\KG(n)$ is a core if and only if $n \neq 4$.
\end{thm}

Gaster, Greene, and Vlamis were interested in the total Kneser graph because of a natural homomorphism from the curve graph of a punctured sphere. If $\Sigma_n$ is a sphere with $n$ points removed, then the \emph{curve graph} $\C(\Sigma_n)$ is the graph with a  vertex for each isotopy class of essential simple closed curves on $\Sigma_n$ with edges if two classes have disjoint representatives. The curve graph is an infinite graph that plays a pivotal role in the rich theory of surfaces and their mapping class groups; see for example \cite{MM_Curve_Graphs_1, BCM_Ending_Laminations_2, Ivanov_Automorphisms, Brock_Weil-Peterson, BBF_Projection_complexes}. 
Curve graphs provide a powerful tool for translating topological problems about surfaces into combinatorial problems about graphs.

If you label the punctures of $\Sigma_n$ with $1,\dots,n$, then each essential curve $\gamma$ cuts $\Sigma_n$ into two pieces and the labels of the punctures on each piece of $\Sigma_n \setminus \gamma$ are invariant under isotopies of $\gamma$. Thus, we get a surjective graph homomorphism $\C(\Sigma_n) \to \KG'(n) = \KG(n) \setminus \KG(n,1)$. From this point of view, Theorem \ref{into:core} is a purely combinatorial analogue of the fact that curve graphs are cores \cite{I_endomorphisms_of_curvegraphs_1,I_endomorphisms_of_curvegraphs_2}.

This homomorphism means the chromatic number of $\KG'(n)$ is an upper bound for the chromatic number of $\C(\Sigma_n)$. While Gaster, Greene, and Vlamis give good lower bounds on the chromatic number of $\C(\Sigma_n)$ the exact chromatic number is still unknown. It is therefore natural to  wonder if $\C(\Sigma_n)$ has a subgraph isomorphic to $\KG'(n)$. Gaster, Greene, and Vlamis prove the answer is ``no'' in the special case when $n=5$. Their proof uses that $\KG(5)$ is a core and is readily generalizable to higher $n$ now that we have Theorem \ref{into:core}.

\begin{cor}
    For $n \geq 5$, the curve graph of the $n$-punctured sphere does not contain a subgraph isomorphic to $\KG'(n)$.
\end{cor}

We also investigate the automorphism group of $\KG(n)$ and show it is naturally isomorphic to $\sym(n) \times \sym(n)$ where the first factor is acting on the subgraph $\KG(n,1)$ and the second factor is acting on $\KG'(n) = \KG(n) \setminus \KG(n,1)$. This mirrors the case of the classical Kneser graph where the automorphism group is known to be $\sym(n)$.

\begin{thm}\label{into:automorphism}
    For $n \geq 5$, the natural action of $\sym(n) \times \sym(n)$ on $\KG(n)$ gives an isomorphism $\sym(n) \times \sym(n) \cong \aut(\KG(n))$
\end{thm}

Our proofs of Theorem \ref{into:core} and Theorem \ref{into:automorphism} both reduce to the known analogous results for classical Kneser graphs. In both cases, this is accomplished by showing that the maps (endomorphism or automorphism respectively) must preserve the strata $\KG(n,k)$. In the case of automorphisms, this is accomplished by counting the degree of each vertex while in the case of endomorphisms we examine the \emph{fractional chromatic number} of the star of each vertex.

\section*{Acknowledgments} This work was completed during summer research funded by Swarthmore College's Division of Natural Sciences and Engineering. We are grateful to Edgar Bering for recommending that we look at Gaster, Greene, and Vlamis's paper on coloring curve graphs as inspiration for this work. We are also thankful to Nick Vlamis for providing the initial opinion on the viability of various open problems in the area. We thank Pat Devlin for helping us navigate the combinatorics literature.

\section{Definitions and Notations}

\subsection{Graphs} When $X$ is a graph with vertex set $V(X)$ and edge set $E(X)$, we use $x \in X$ to denote $x \in V(X)$. An \emph{independent set} in $X$ is a subset of vertices where no two vertices are joined by an edge. 

Given a subset $S \subseteq V(X)$, the subgraph \emph{spanned} by $S$ is the graph with  vertex set $S$ and with $\{x,y\} \subseteq S$ getting an edge if and only if $\{x,y \} \in E(X)$.  Given $x \in V(X)$, we let $X - x$ denote the subgraph  spanned by all vertices except $x$. The \emph{link} of a vertex $\lk(x)$ is the subgraph of $X$ spanned by the set of vertices that are joined by an edge to $x$. The \emph{star} of a vertex $\st(x)$ is the subgraph spanned by $\lk(x)$ plus $x$. Given two graphs $X,Y$, the \emph{join} $X\ast Y$ is the graph obtained by joining each vertex of $X$ to each vertex of $Y$ by an edge. The star $\st(x)$ is always the join of $\lk(x)$ with the vertex $x$. 

A graph \emph{homomorphism} $\phi \colon X \to Y$ is a map from $V(X)$ to $V(Y)$ that preserves edges. If $Z$ is a subgraph of $X$, then $\phi(Z)$ denotes the subgraph of $Y$ with vertex set $ \phi(V(Z))$ and edges between $\phi(x)$ and $\phi(y)$ if $\{x,y\} \in E(Z)$. If $X$ is equal to $Y$, then we call a homomorphism an \emph{endomorphism}.  An \emph{automorphism} is an endomorphism that is bijective on  the vertices and edges.  We use $\aut(X)$ to denote the group of automorphisms of the graph $X$. We say the graph $X$ is a \emph{core} if every endomorphism of $X$ is an automorphism. Basic examples of cores include the complete graphs and cycles with an odd number of vertices.

\subsection{Total Kneser Graphs}
Henceforth $n$ and $k$ will be positive integers. Given a set $\Omega$ of size $n$, define a \emph{split} of $\Omega$ to be an unordered pair  $\{A,A^c\}$ where $A$ and $A^c$ are both non-empty subsets of $\Omega$. We will use $A|A^c$ to denote a split. The \emph{total Kneser graph}, $\KG(n)$, is the graph whose vertices are splits of a fixed set of size $n$ with an edge between the splits $A|A^c$ and $B|B^c$ if one of the following mutually exclusive options holds: $A\subseteq B$, $B \subseteq A$, $A \subseteq B^c$, or $A^c \subseteq B$. For concreteness, we will usually use the set $[n] = \{1,\dots, n\}$ as the set of size $n$ used to define $\KG(n)$.

For a positive integer $k \leq \lfloor\frac{n}{2} \rfloor$, the \emph{$k$-stratum} of $\KG(n)$ is the subgraph spanned by the vertices $A|A^c$ where $A$ has $k$ elements. We denote the $k$-stratum by $\KG(n,k)$. We use $\KG'(n)$ to denote the subgraph of $\KG(n)$ spanned by all the strata except $\KG(n,1)$. We call $\KG'(n)$ the \emph{truncated total Kneser graph}. Since one element sets in $[n]$ are either nested into or disjoint from every other non-empty subset of $[n]$, the first stratum $\KG(n,1)$ is isomorphic to the complete graph $K_n$ and $\KG(n)$ is the join of $\KG'(n)$ and $\KG(n,1)$.

When $k<n/2$, the stratum $\KG(n,k)$ is isomorphic to the classical \emph{Kneser graph}, $K_{n:k}$. The classical Kneser graph  $K_{n:k}$ is the graph whose vertices are $k$ element subsets of $[n]$ with edge between subsets if the subsets are disjoint. When $n$ is even and $k=n/2$, then $\KG(n,k)$ is the graph with   $\frac{1}{2}{n \choose n/2}$ vertices and no edges while $K_{n:k}$ is the graph that is the disjoint union of $\frac{1}{2}{n \choose n/2}$ edges.  We use $\strata(x) = k$ to denote that $x \in \KG(n)$ is an element of the stratum $\KG(n,k)$.

A classic fact about $\KG(n,k) = K_{n:k}$ is that the graph is a core when $n>2k$.

\begin{thm}[{\cite[Theorem 7.9.1]{GR_Alg_Graph_Theory_Book}}]
    If $n >2k$, then $\KG(n,k)$ is a core.
\end{thm}

\subsection{Fractional Chromatic Numbers} The fractional chromatic number is a refinement of the classical chromatic number. Conceptually, one can think of a fractional coloring as a coloring of a graph where each vertex can  be assigned multiple colors, but adjacent vertices cannot have any colors in common. The fractional chromatic number, then measures the minimal ratio of total colors used to the number of colors per vertex.

To give the formal definition, let $\I(X)$ be the set of all independent sets in the graph $X$ and let $\I(X,x)$ be the subset of $\I(X)$ consisting of independent sets that contain $x \in X$. A \emph{fractional coloring} of $X$ is a function $f \colon \mathcal{I}(X) \to [0,\infty)$ so that  $\sum_{S \in \I(X,x)} f(S) \geq 1$ for each $x \in X$. The \emph{weight} of a fractional coloring is $\sum_{S \in \I(X)} f(S)$ and the \emph{fractional chromatic number} $\fc(X)$ is the minimal weight of a fractional coloring of $X$.  

The fractional chromatic number behaves similarly to the usual chromatic number with respect to joins and homomorphisms.

\begin{lem}
Let $X$ and $Y$ be graphs.
\begin{itemize}
    \item If $\phi \colon X \to Y$ is a graph homomorphism, then $\fc(X) \leq \fc(Y)$.
    \item     $\fc(X \ast Y) = \fc(X) + \fc(Y)$.
\end{itemize}
\end{lem}

\begin{proof}
The first bullet is immediate from the definition and the fact that the pre-images of independent sets under homomorphisms are independent sets.

    Since $\I(X\ast Y) = \I(X) \cup \I(Y)$, we have $\fc(X \ast Y) \leq \fc(X) + \fc(Y)$ simply by selecting a minimal weight fractional coloring of $X$ and of $Y$ separately.    Now, suppose $f$ is a minimal weight fractional coloring of $X \ast Y$. Then $f \vert_X$ and $f\vert_Y$ are fractional colorings of $X$ and $Y$ respectively. Thus we have $$ \fc(X \ast Y) = \sum_{S \in \I(X \ast Y)} f(S) = \sum_{S\in \I(X)} f(S) + \sum_{S \in \I(Y)}f(S) \geq \fc(X) + \fc(Y).$$
\end{proof}

The fractional chromatic number of the classical Kneser graph  $\KG(n,k)$ is well known to be   $n/k$ \cite[Corollary 7.7.3]{GR_Alg_Graph_Theory_Book}. Gaster, Greene, and Vlamis proved that putting the minimal weight fractional coloring on each stratum of $\KG(n)$ yields the fractional chromatic number of $\KG(n)$. 

\begin{thm}[\cite{GGV_Coloring_curves,GR_Alg_Graph_Theory_Book}]\label{thm:frac_chrom}
Let $n > 2k$ and $H_m$ denote the $m$th harmonic number $H_m = \sum_{i=1}^m \frac{1}{i}$. We have
$$\fc(\KG(n,k)) = n/k, \ \fc(\KG(n)) = n H_{\lfloor n/2 \rfloor} - p(n) ,  \text{ and } \fc(\KG'(n)) = n H_{\lfloor n/2 \rfloor} - n -  p(n),$$
where $p(n) =0$ if $n$ is odd and $1$ if $n$ is even.
\end{thm}

\subsection{Permutations} We let $\sym(n)$ denote the group of bijections of $[n]$. Each permutation $\phi \in \sym(n)$ acts on the set of subsets of $[n]$ while preserving the size, nesting relations, and disjointness. Thus, each permutation  $\phi \in \sym(n)$ induces an element of $\aut(\KG(n))$. This action of $\phi \in \sym(n)$ restricts to an action on each stratum $\KG(n,k)$ individually and hence gives an action on $\KG'(n)$ as well. For the classical Kneser graph $\KG(n,k)$, these permutations are known to be all of the automorphisms.

\begin{thm}[{\cite[Corollary 7.8.2]{GR_Alg_Graph_Theory_Book}}]
    Let $n >2k$. The natural action of $\sym(n)$ on $\KG(n,k)$ induces an isomorphism $\sym(n) \cong \aut(\KG(n,k)).$
\end{thm}

\section{Automorphisms of the Total Kneser Graph}\label{sec:automorphisms}
 Since $\KG(n)$ is the join $\KG(n,1) \ast \KG'(n)$,  any pair of permutations $(\sigma,\tau) \in \sym(n) \times \sym(n)$ will give an automorphism by having $\sigma$ act on $\KG(n,1)$ and $\tau$ act on $\KG'(n)$.  We will prove that this action induces an isomorphism  $\sym(n) \times \sym(n) \cong \aut(\KG(n))$.

 First we calculate the degree of a vertex in terms of its stratum. This proves that automorphisms of $\KG(n)$ must preserve the strata.

\begin{lem}\label{lem:degree}
    Let $x \in \KG(n)$. If $\strata(x) = k$, then $\deg(x) = 2^{k}+2^{n-k} - 4$. Thus, every automorphism of $\KG(n)$ preserves the strata of $\KG(n)$.
\end{lem}

\begin{proof}
    Let $x=A|A^c$ with $|A|=k,$ so $|A^c|=n-k.$  If $B|B^c$ is adjacent to $A|A^c$, then exactly one of $B\subseteq A$,  $B^c \subseteq A^c$, $B^c\subseteq A$, or $B \subseteq A^c$ holds. Thus, there is a bijection between vertices adjacent to $A|A^c$ and the set 
    $$ \{B \subseteq A : B \neq \emptyset, B \neq A\} \cup \{ B \subseteq A^c : B \neq  \emptyset, B \neq A^c\}.$$
    Thus $\deg(x) = 2^{k}+2^{n-k} - 4$.

    Since $2^{k}+2^{n-k} - 4$ is a strictly decreasing function between $k=1$ and $k = \lfloor n/2 \rfloor$ two vertices can have the same degree if and only if they are in the same stratum. Thus, every automorphism of $\KG(n)$ preserves each stratum.
\end{proof}

Since $\aut(\KG(n,k)) \cong \sym(n)$ for $n >2k$, Lemma \ref{lem:degree} ensures that automorphisms of $\KG(n)$ restrict to a permutation on each stratum (except possibly the $n/2$-stratum). The goal now is to prove that the same permutation is induced on each stratum of $\KG'(n)$. To do this,  we examine the  \emph{$k$-restricted link} of a vertex $x \in \KG(n)$ , which we define to  be the subgraph $$L_k(x) = \lk(x) \cap \KG(n,k).$$  Our next lemma shows that a vertex on $\KG'(n)$ is determined by its restricted link.

\begin{lem}\label{lem:restricted_link_determine_vertices}
    Let $x,y\in \KG(n,k+1)$ for $2\leq k \leq n/2-1$. If $L_k(x) = L_k(y)$, then $x = y$. 
\end{lem}

\begin{proof}
    Let $x=A|A^c$ be a vertex of $\KG(n,k+1)$. The vertices in $\KG(n,k)$ that are adjacent to $x=A|A^c$ are exactly the vertices $$ \{C|C^c : |C|=k  \text{ and } C \subseteq A  \text{ or } C \subseteq A^c\}.$$ 
    
  Now suppose that $y = B|B^c \in \KG(n,k+1)$ has  $L_k(x)=L_k(y).$
    Then the collection of subsets of size $k$ contained in either $A$ or $A^c$ is the same as the collection of subsets of size $k$ contained in either $B$ or  $B^c$.
  
    For the purposes of contradiction, assume $\{A,A^c\} \neq \{B,B^c\}.$
    Without loss of generality, assume $|A|=|B|=k+1$. Thus, $A \setminus B \neq \emptyset$ as $A \subseteq B$ would force $A=B$.  Fix $a_0 \in A \setminus B$.

    Let $a \in A \setminus \{a_0\}$. Because $k >1$, there is a subset $C \subseteq A$ with $|C| = k$ and $a_0,a \in C$. Now $C$ must be contained in either $B$ or $B^c$ because $L_k(x)=L_k(y)$. However, $a_0 \notin B$ implies that $C \subseteq B^c$. Thus $a \in B^c$. 

   We now make a similar argument with $b \in B^c \setminus \{a_0\}$ and $a_0$. The set $\{b,a_0\}$  is contained in a $k$ element subset of $B^c$ and this subset must also be contained in $A$ because $L_k(x)=L_k(y)$ and $a_0 \notin A^c$. Hence $b \in A$. Together this implies $A = B^c$ which contradicts $\{A,A^c\} \neq \{B,B^c\}.$
\end{proof}

We now classify the automorphisms of $\KG'(n)$ and $\KG(n)$.

\begin{thm}
    Let $n \geq 5$. The action of $\sym(n)$ on $[n]$ induces an isomorphism $\sym(n)  \cong \aut(\KG'(n))$.
\end{thm}

\begin{proof}
   Let $\phi \in \aut(\KG'(n))$.  Lemma \ref{lem:degree} implies that $\phi$  preserves the strata of $\KG'(n)$, because the degree of a vertex in $\KG'(n)$ is exactly $n$ less than the degree of a vertex in $\KG(n)$.
   Define $\phi_k \in \aut(\KG(n,k))$ to be the restriction of $\phi$ to $\KG(n,k)$. If $k <n/2$, then the inclusion $\sym(n)$ into $\aut(\KG(n,k))$ is an isomorphism. Thus for each $1<k<n/2$, $\phi_k$ will act by some permutation  on $\KG(n,k)$. We  show that if $\phi_k =\sigma \in \sym(n)$, then $\phi_{k+1} = \sigma$ as well.

   Let $x= A|A^c \in \KG(n,k+1)$ and let $\sigma \in \sym(n)$ so that $\phi_k = \sigma$. Because $\phi$ is an automorphism $$\phi
    (L_k(x))= L_k(\phi(x)).$$
   Since $\phi_k = \sigma$, we have $$ L_k(\phi(x)) = \phi
    (L_k(x))= \phi_k(L_k(x)) = \sigma(L_k(x)) = L_k(\sigma(A)|\sigma(A)^c).$$
     By Lemma \ref{lem:restricted_link_determine_vertices}, a vertex in $\KG(n, k+1)$ is uniquely determined by its restricted link in $\KG(n,k)$. Hence, $$\phi(x)=\sigma(A)|\sigma(A)^c.$$
    Therefore, $\phi_{k+1}$ induces the permutation $\sigma$ on $\KG(n,k+1)$.

    Using induction with the base case of $k=2$, we now have that $\phi$ acts on every stratum of $ \KG'(n)$ by the same permutation. Thus the inclusion of $\sym(n)$ to $\aut(\KG'(n))$ is an isomorphism.
\end{proof}

\begin{cor}
    For $n \geq 5$, the inclusion $\sym(n) \times \sym(n)  \to \aut(\KG(n))$  where the first factor acts on $\KG(n,1)$ and the second factor acts on $\KG'(n)$ is an isomorphism.
\end{cor}

\begin{proof}
    Since the automorphisms of $\KG(n)$ preserve the strata, they also preserve $\KG'(n)$. Combining this with the fact that $\KG(n)$ is the join of $ \KG'(n)$ and $\KG(n,1)$, we have $\aut(\KG(n)) \cong \aut(\KG'(n)) \times \aut(\KG(n,1)).$ Since $\KG(n,1)$ is the complete graph on $n$ vertices, the result follows from the fact that $\sym(n) \cong \aut(\KG'(n))$.
\end{proof}

We conclude by noting that $\aut(\KG(n))$ does not equal $\sym(n) \times \sym(n)$ when $2 \leq n \leq 4$ (in the $n=1$ case $\KG(1)$ is empty).  $\KG(2)$ is a single vertex group, so it has a single automorphism. $\KG(3)$ is a 3-cycle, so $\aut(\KG(3))  \cong \sym(3)$. For $\KG(4)$, $\KG'(4) = \KG(4,2)$ is the graph with 3 vertices and no edges. Thus,  $\aut(\KG(4)) \cong \sym(4) \times \sym(3)$. In this case, the kernel of the $\sym(4)$ action on $\KG'(4) = \KG(4,2)$ is the Klein four subgroup of $\sym(4)$. 

\section{Endomorphisms of Total Kneser Graphs}

We now turn to our main result of proving $\KG(n)$ and $\KG'(n)$ are cores. The heart of our proof is to show that for any homomorphism $\phi$ we have  $$\phi(\KG(n,k)) \subseteq \KG(n,k)$$ for each stratum. This will ensure that every homomorphism of $\KG(n)$ restricts to a homomorphism on the strata. We can then invoke the fact that $\KG(n,k)$ is a core for nearly all $k$ to prove that $\phi$ is bijective.

Our first proposition says that if a homomorphism preserves the strata up to stratum $k-1$, then it cannot lower the vertices in stratum $k$.

\begin{prop}\label{prop:inductive_step}
      Let $k \in \{2,\dots, \lfloor n/2 \rfloor\}$ and $\phi \colon \KG(n) \to \KG(n)$ be a graph homomorphism. If $\phi(\KG(n,j)) = \KG(n,j)$ for all integers $1\leq j \leq k-1$, then $\strata(\phi(x)) \geq \strata(x)$ for all $x \in \KG(n,k)$.
\end{prop}

\begin{proof}
    Let $x=A|A^c\in \KG(n,k)$.  Suppose for contradiction that $\phi(x)=y=B| B^c\in \KG(n,j)$ for some $j\leq k-1$.

Since $\phi$ preserves $\KG(n,j)$ and each $\KG(n,j)$ is a core, the restriction $\phi|_{\KG(n,j)}$
is an automorphism. Thus, there is $\sigma \in \sym(n)$ so that $\phi|_{\KG(n,j)} = \sigma$.  We can therefore assume that $\phi$ fixes every vertex in $\KG(n,j)$ by replacing $\phi$ with $\phi \circ \sigma^{-1}$.

Since $\phi(x)=\phi(y)=y$, it is impossible for $x$ and $y$ to be joined by an edge of $\KG(n)$. Thus each of the  intersections $A\cap B, A\cap B^c, A^c\cap B, A^c\cap B^c$
must  be non-empty.  We can therefore choose  $C\subseteq A$ so that $|C| = j$,
$C\cap B\neq \emptyset$ and $C \cap B^c\neq\emptyset.
$ Let $z=C| C^c\in \KG(n,j).$ Since $C\subseteq A$, $x$ and $z$ are joined by an edge. Thus, $\phi(z) = z$ and $\phi(x) = y$ are also joined by an edge. However, this creates a contradiction because  $$C\cap B\neq \emptyset, C \cap B^c\neq\emptyset, C^c\cap B\supseteq A^c\cap B\neq \emptyset, \text{ and }  C^c\cap B^c\supseteq A^c\cap B^c\neq \emptyset $$ means that $z$ and $y$ are not joined by an edge.  Thus $\strata(\phi(x)) < \strata(x)$ is impossible. 
\end{proof}

\begin{cor}\label{cor:inductive_step}
    Let $k \in \{2,\dots, \lfloor n/2 \rfloor\}$ and $\phi \colon \KG(n) \to \KG(n)$ be a graph homomorphism. Suppose that for  all $x\in \KG(n,k)$, we have $\strata(\phi(x)) \leq \strata(x)$. In this case, if $\phi(\KG(n,j)) = \KG(n,j)$ for all integers $1\leq j \leq k-1$, then $\phi(\KG(n,k)) \subseteq \KG(n,k)$.     Moreover, if $k < n/2$ then we can conclude $\phi(\KG(n,k)) =\KG(n,k)$.
\end{cor}

\begin{proof}
    The first statement follows immediately from Proposition \ref{prop:inductive_step} plus the assumption that $\strata(\phi(x)) \leq \strata(x)$. The moreover statement is true because $\KG(n,k)$ is a core  for $1 \leq k <n/2$. Thus if $\phi\vert_{\KG(n,k)}$ has image in $\KG(n,k)$, then $\phi\vert_{\KG(n,k)}$ must be an automorphism of $\KG(n,k)$, which ensures $\phi(\KG(n,k)) =\KG(n,k)$.
\end{proof}

In light of Corollary \ref{cor:inductive_step}, if we can show that the stratum of a vertex cannot decrease, then we can inductively prove that the strata are fixed. To prevent the lowering of the stratum, we look at the fractional chromatic number of the star of each vertex.

\begin{lem}\label{lem:link_structure}
    Let $x \in \KG(n)$. If $\strata(x) = k \leq n/2$, then $$\fc(\st(x)) = \fc(\KG(k+1)) + \fc(\KG(n-k+1)) -1.$$
\end{lem}

\begin{proof}
    For this proof, we use an isomorphic formulation of the total Kneser graph. Consider the vertices of $\KG(n)$ as the non-empty subsets of $[n]$ that do not contain the  element $n$ with edges between subsets $A$ and $B$ if either $A\subseteq B$, $B \subseteq A$, or $A \cap B = \emptyset$. The assignment $A \to A|A^c$ gives an isomorphism with our previous definition of $\KG(n)$. 

    Let $x = A \in \KG(n)$ with $\strata(x) = k$. Because $\aut(\KG(n))$ acts transitively on each stratum, we can assume without loss of generality that $A = \{1,\dots, k\}$. We  partition the vertices of $\lk(x)$ into two sets:
    \begin{itemize}
        \item $\nestedset$ the set of all non-empty proper subsets of $[k]$,
        \item $\shiftedset$ the set of vertices of $\lk(x)$ that either contain $[k]$ or are disjoint from $[k]$.
    \end{itemize}

 Let $X_\nestedset$ and $X_\shiftedset$ be the subgraphs of $\KG(n)$ spanned by $\nestedset \cup \{x\}$ and $\shiftedset \cup \{x\}$ respectively. To see that $\st(x) = X_\nestedset \ast (X_\shiftedset-x)$, let $B \in \nestedset$ and $C \in \shiftedset$. If $C$ is disjoint from $[k]$, then $C \cap B = \emptyset$. If $C \cap [k] \neq \emptyset$, then $[k] \subseteq C$, which implies $ B \subseteq C$. In all cases, $B$ and $C$ are joined by an edge of $\KG(n)$. Thus, every vertex of $\nestedset$ is joined by an edge to every vertex of $\shiftedset$. We will now argue that $X_\nestedset = \KG(k+1)$ and $X_\shiftedset$ is isomorphic to $\KG(n-k+1)$.
   
    First, observe that $\nestedset \cup \{x\}$  is the set of all non-empty subsets of $[k+1]$ that do not contain $k+1$. Because edges in $\KG(n)$ are for nesting or disjointness, $X_\nestedset$  is  precisely $\KG(k+1)$. 

    For $X_\shiftedset$, we define the following map $\theta\colon V(X_\shiftedset) \to V(\KG(n-k+1))$.
    \begin{enumerate}
        \item If $B \in V(X_\shiftedset)$ is disjoint from $[k]$, then $\theta(B) = \{ i \in [n-k+1] : i + k \in B\}$, aka the subset obtained by subtracting $k$ from each element of $B$.
        \item If  $B \in V(X_\shiftedset)$ contains $[k]$, then $\theta(B) = \{i \in [n-k+1] : i+k \in B\} \cup \{n-k\}$, aka the subset obtained by subtracting $k$ from each element of $B$ and then adding in the element $n-k$.
    \end{enumerate}
    This map is bijective as the first option achieves all the non-empty subsets of  $[n-k+1]$ that do not contain either $n-k$ or $n-k+1$ (because none of the $B \in\shiftedset$ contain $n$ or $n+1$) while the second option achieves all the subsets of $[n-k+1]$ that do not contain $n-k+1$ but do contain $n-k$.

    To verify that $\theta$ preserves the edges, let $B,C$ be elements of $\shiftedset \cup \{[k]\}$ that are either nested or disjoint. If $B\cap C = \emptyset$, then at most one of $B$ or $C$ can contain $[k]$. Thus, $\theta(B) \cap \theta(C) = \emptyset$. If $B \subseteq C$, then every $i \in \theta(B)$ is either $n-k$ or equal to $b-k$ for some $b\in B$. In the latter case, $i \in \theta(C)$ because $b \in C$. If $i=n-k$, then $i = n-k \in \theta(C)$ because  $[k]\subseteq B$ implies $[k] \subseteq C$. In both cases, $B \subseteq C$ implies $\theta(B) \subseteq \theta(C)$.

    Finally, to see that $\theta$ also preserves  non-edges, observe that when  $B,C \in \shiftedset \cup \{[k]\}$ overlap but are not nested then $\theta(B)$ and $\theta(C)$ also overlap without nesting.

    We have now shown that $X_\nestedset$ and $X_\shiftedset$ are isomorphic to $\KG(k+1)$ and $\KG(n-k+1)$ respectively. This implies that $$\fc(X_\nestedset) = \fc(\KG(k+1)) \text{ and } \fc(X_\shiftedset - x) = \fc(\KG(n-k+1)) -1,$$ which completes the proof (the minus 1 follows from the fact that every vertex of $X_\shiftedset - x$ is joined by an edge to $x$). 
\end{proof}

Since homomorphisms must send stars into stars, the fractional chromatic number of $\st(x)$ is an invariant of $x$ that cannot decrease under homomorphisms. The next lemma uses the fractional chromatic number of the star to prove that the strata can only increase in a very special case.

\begin{lem}\label{lem:strata_do_not_increase}
    Let $\phi \colon \KG(n) \to \KG(n)$ be a graph homomorphism.
    If $\strata(\phi(x)) > \strata(x)$, then $4|n$, $\strata(x) = \frac{n}{2}-1$, and $\strata(\phi(x)) = \frac{n}{2}$. 
\end{lem}

\begin{proof}
Let $x\in \KG(n)$ with $\strata(x) = k$. Let $\strata(\phi(x))=\ell$. Assume $\ell>k$. The restriction of $\phi$ to the star of $x$ gives a graph homomorphism
$
\phi|_{\st(x)}\colon\st(x)\to\st(\phi(x)).$
Therefore,
\[
\fc(\st(x))\leq\fc(\st(\phi(x))).
\]
For
$
m \in \{1,\dots,\lfloor\frac{n}{2}\rfloor\}$
define, 
\[
F(m)= \fc(\KG(m+1))+\fc(\KG(n-m+1))-1.
\]
By Lemma \ref{lem:link_structure}, if $y \in \KG(n)$ has $\strata(y) = m$, then $F(m) = \fc(\st(y))$. Thus, $F(k) \leq F(\ell)$.

Our next step is to prove that $F(k) \geq F(\ell)$ as well. To start, we check that $$\fc(\KG(m+1))-\fc(\KG(m)) = \harm{m/2}+1. $$  When $m$ is even, we have $\harm{\frac{m+1}{2}} = \harm{\frac{m}{2}}$. Thus $$\fc(\KG(m+1))-\fc(\KG(m)) = (m+1)\harm{\frac{m+1}{2}} - \left( m\harm{\frac{m}{2}} -1\right) = \harm{\frac{m}{2}} +1 .$$
When $m$ is odd, we have $$ \harm{\frac{m+1}{2}} = \frac{2}{m+1} + \harm{\frac{m-1}{2}} = \frac{2}{m+1} + \harm{\frac{m}{2}}.$$ Thus, 
\begin{align*}
    \fc(\KG(m+1))-\fc(\KG(m)) =& (m+1)\harm{\frac{m+1}{2}} -1  - m\harm{\frac{m}{2}} \\
    =& 2 +(m+1)\harm{\frac{m}{2}} - 1-m\harm{\frac{m}{2}}\\
    =&  \harm{\frac{m}{2}} +1 .\\
\end{align*}

Applying this to $F$ we have 
\begin{align*}
F(m+1) - F(m)
&=
\fc(\KG(m+2)) - \fc(\KG(m+1))+ \fc(\KG(n-m)) -\fc(\KG(n-m+1))\\
&=
H_{\left\lfloor\frac{m+1}{2}\right\rfloor} -H_{\left\lfloor\frac{n-m}{2}\right\rfloor},
\end{align*}
for every $m \in \{1,\dots,\lfloor\frac{n}{2}\rfloor-1\}$.
Since $m<\frac{n}{2}$, we have
$
\left\lfloor\frac{m+1}{2}\right\rfloor \leq \left\lfloor\frac{n-m}{2}\right\rfloor,
$ which implies $F(m+1)\leq F(m)$ 
as the harmonic numbers are strictly increasing.
Since $\ell>k$, it follows that
$
F(k)\geq F(\ell).$
Combining this with $F(k)\leq F(\ell)$ we obtain
$
F(k)=F(\ell).$ 

Now $
F(k)=F(\ell)$ gives

\[
0= F(\ell) - F(k)
=
\sum_{m=k}^{\ell-1}\bigl(F(m+1)-F(m)\bigr)
\]
and every term in this sum is non-positive, so every term must be zero. In particular,
$F(k)$ must equal $F(k+1).$ However $$F(k) - F(k+1) = H_{\left\lfloor\frac{n-k}{2}\right\rfloor}
-
H_{\left\lfloor\frac{k+1}{2}\right\rfloor},$$ and the difference of two harmonic numbers is zero only when the harmonic numbers are equal. Thus, 
$
\left\lfloor\frac{n-k}{2}\right\rfloor
=
\left\lfloor\frac{k+1}{2}\right\rfloor$, and this holds if and only if $4|n$ and   $k=\frac{n}{2}-1 = \strata(x).$ This means that $\ell$ must be $k+1 = \frac{n}{2}$.
\end{proof}

We now prove that $\KG(n)$ is a core. Lemma \ref{lem:strata_do_not_increase} ensures that homomorphisms preserve the first stratum. Corollary \ref{cor:inductive_step} will then inductively show that the higher strata are preserved except for the highest or second to highest stratum in the case of $n$ being even. We handle those two cases  with a variation of  the proof of Proposition \ref{prop:inductive_step}.

\begin{thm}\label{thm:KG_is_core}
    The total Kneser graph $\KG(n)$ is a core for all $n \geq 5$.
\end{thm}

\begin{proof}
Let $\phi\colon \KG(n) \to \KG(n)$ be a graph homomorphism. It will suffice to prove that $$\phi(\KG(n,k)) = \KG(n,k)$$ for all $k \in \{1,\dots,\lfloor n/2\rfloor\}.$ 

Lemma \ref{lem:strata_do_not_increase} ensures that $\phi(\KG(n,1)) \subseteq \KG(n,1)$ for all $n \geq 5$.  Since $\KG(n,1)$ is a core, this implies $\phi(\KG(n,1)) = \KG(n,1)$. We can now use Lemma \ref{lem:strata_do_not_increase} plus Corollary \ref{cor:inductive_step} to inductively to show that $\phi(\KG(n,k)) = \KG(n,k)$ for all $k \in \{1,\dots, \lfloor n/2 \rfloor-2\}$. We now divide into cases based on $n$.

\emph{Case 1:} Assume $n$ is odd. In this case,  Lemma \ref{lem:strata_do_not_increase} and  Corollary \ref{cor:inductive_step}  show that $\phi(\KG(n,k)) = \KG(n,k)$ for all $k \in \{1,\dots,\lfloor n/2\rfloor\}$  and we are finished.

\emph{Case 2:} Assume $n$ is even, but not divisible by 4. Let $n = 2m$. In this case,  Corollary \ref{cor:inductive_step}  and Lemma \ref{lem:strata_do_not_increase} show that $\phi(\KG(n,k)) = \KG(n,k)$ for all $k \in \{1,\dots,m-1\}$. In particular, if $\phi_k$ is the restriction of $\phi$ to $\KG(n,k)$, then $\phi_k$ is an automorphism of $\KG(n,k)$ for $k < m$. It remains to show that $\phi(\KG(n,m)) = \KG(n,m).$ 

Because $n$ is not divisible by $4$, Corollary \ref{cor:inductive_step} establishes that $\phi(\KG(n,m)) \subseteq \KG(n,m)$. Let $x,y \in \KG(n,m)$ be distinct vertices. Since $\phi_{m-1}$ is an automorphism of $\KG(n,m-1)$,  there is $\sigma \in \sym(n)$ so that $\phi_{m-1} = \sigma$. Moreover,   $$L_{m-1}(\phi(x)) = \sigma(L_{m-1}(x)) \text{ and } L_{m-1}(\phi(y)) = \sigma(L_{m-1}(y))$$ where $L_{m-1}(\cdot)$ is the restricted link from Section \ref{sec:automorphisms}. Since the restricted links determine vertices  (Lemma \ref{lem:restricted_link_determine_vertices}), $x \neq y$ implies $\sigma(L_{m-1}(x)) \neq \sigma(L_{m-1}(y))$. Thus, $L_{m-1}(\phi(x)) \neq L_{m-1}(\phi(y))$. But this implies $\phi(x) \neq \phi(y)$ by Lemma \ref{lem:restricted_link_determine_vertices}. Hence, $\phi_m$ is injective which implies $\phi(\KG(n,m)) = \KG(n,m)$.

\emph{Case 3:} Assume $n$ is divisible by 4. Let $n = 2m$. We first show that $\strata(\phi(x))\leq m-1$ for all $x \in \KG(n,m-1)$.

Since we know $\phi$ preserves the stratum $\KG(n,m-2)$ and $\KG(n,m-2)$ is a core, we can assume that $\phi$ fixes each vertex of $\KG(n,m-2)$ by precomposing with the appropriate permutation.  Let $x \in \KG(n,m-1)$ and suppose for contradiction that  $\phi(x) \in \KG(n,m)$. 

Let $A|A^c = x$ and $B|B^c=\phi(x)$. Without loss of generality, we can assume $|A\cap B| \leq |A\cap B^c|$ by switching $B$ and $B^c$.

    First, suppose that $A\cap B \neq \emptyset$. Then both $A\cap B$ and $A\cap B^c$ are nonempty because  $|A\cap B| \leq |A\cap B^c|$. Since $m \geq 4$, we can choose a subset $C\subseteq A$ with $|C|=m-2$ such that
$ C\cap B\neq\emptyset$
and
$C\cap B^c\neq\emptyset.$ Let $z=C|C^c \in \KG(n,m-2)$.

Since $C\subseteq A$,  $x$ and $z$ are joined by an edge. However,  $C^c\cap B\supseteq A^c\cap B\neq\emptyset$ means that $z = \phi(z)$ and $\phi(x) = B|B^c$ are not joined by an edge as $C\cap B$, $C^c \cap B$, and $C \cap B^c$ are all non-empty. This violation of $\phi$ being a homomorphism means that $\strata(\phi(x)) \neq m$.

Now suppose that  $A\cap B = \emptyset$. Then $A\subseteq B^c$, and $B^c\setminus A$ consists of a single element. Choose a subset $C\subseteq A^c$ with $|C|=m-2$ such that $C$ contains the element of $B^c\setminus A$ and at least one, but not all, of the elements of $B$. Let $
z=C|C^c.$ Since $C \cap B \neq \emptyset$, $C \cap B^c \neq \emptyset$, and $C^c \cap B \neq \emptyset$, $z = \phi(z)$ is not joined by an edge to $\phi(x) = B|B^c$. However, $C \subseteq A^c$ means that $z$ and $x$ are joined by an edge. This contradiction again means that $\strata(\phi(x)) \neq m$.

In all cases, we have shown $\strata(\phi(x)) \leq \strata(x) = m-1$. We can now apply Corollary \ref{cor:inductive_step} to conclude that $ \phi(\KG(n,m-1)) = \KG(n,m-1)$.

The last remaining step is to verify that $\phi(\KG(n,m)) = \KG(n,m)$. Since $\strata(\phi(x)) \leq \strata(x) = m$ is automatically true for $x \in \KG(n,m),$ Corollary \ref{cor:inductive_step} now tells us that $\phi(\KG(n,m)) \subseteq \KG(n,m)$. We can now finish by repeating the argument from the $n$ even, but not divisible by $4$ case to prove that $\phi\vert_{\KG(n,m)}$ is injective.    
\end{proof}

That $\KG'(n)$ is a core now follows from $\KG(n)$ being a core.

\begin{cor}
    The truncated total Kneser graph $\KG'(n)$ is a core for all $n \geq 5$.
\end{cor}

\begin{proof}
    Let $\phi\colon \KG'(n) \to \KG'(n)$ be a graph homomorphism. Define $\psi \colon \KG(n) \to \KG(n)$ to be the map that is $\phi$ on the vertices of $\KG'(n)$ and the identity on $\KG(n,1)$. Since $\KG(n)$ is a join of $\KG'(n)$ and $\KG(n,1)$, $\psi$ is a graph homomorphism. By Theorem \ref{thm:KG_is_core}, $\psi$ is an automorphism of $\KG(n)$. Thus $\phi = \psi\vert_{\KG'(n)}$ is an automorphism of $\KG'(n)$.
\end{proof}

Theorem \ref{thm:KG_is_core} established that $\KG(n)$ is a core  for $n \geq 5$, which is what we need for our topological application. The remaining cases can be verified individually. For $n=1,2,3$, $\KG(n)$ is a core  because it is either empty, a single vertex, or a $3$-cycle. In contrast, $\KG(4)$ is \emph{not} a core. One can construct many homomorphisms that are not bijective, because $\KG(4)$ is the join of the  $\KG(4,1)$ (which is a complete graph) and $\KG(4,2)$ (which is 3 vertices with no edges). For example, the map that sends all vertices of $\KG(4,2)$ to $1|234$ while sending  $1|234$ to $12|34$ and fixing the rest of $\KG(4,1)$ is a non-injective homomorphism.

\section{Kneser subgraphs of curve graphs of spheres}

Let  $\Sigma_n$ be a sphere with $n$ points removed. We call the removed points \emph{punctures}. A \emph{curve} on $\Sigma_n$ is an isotopy class of a continuous embedding of a circle into $\Sigma_n$. A curve is \emph{essential} if it is not isotopic to a point or a puncture. We will conflate a curve with its image on $\Sigma_n$. The \emph{curve graph} $\C(\Sigma_n)$ is the graph whose vertices are essential curves on $\Sigma_n$ with edges between two curves if they have disjoint representatives.

Label the punctures of $\Sigma_n$ with $1,\dots, n$. Every curve $\gamma$ on $\Sigma_n$ cuts $\Sigma_n$ into two components (up to isotopy) that each contain at least two punctures. The labels of the punctures in each component of $\Sigma_n \setminus \gamma$ are invariant under isotopies of $\gamma$. Thus $\gamma$ determines a well defined split $A_\gamma|A_\gamma^c$ of $[n]$ where $A_\gamma$ and $A_\gamma^c$ are the labels of the punctures on each of the two components of $\Sigma_n \setminus \gamma$. The map $\Phi \colon \C(\Sigma_n) \to \KG'(n)$  by $\Phi(\gamma) = A_\gamma | A_\gamma^c$ is a surjective graph homomorphism because  if two distinct curves $\gamma$ and $\eta$ are disjoint on $\Sigma_n$, then $A_\gamma$ must properly nest into $A_\eta$ or $A_\eta^c$.

We use the fact that $\KG'(n)$ is a core to verify that \cite[Proposition 3.4]{GGV_Coloring_curves} generalizes to all $n \geq 5$.

\begin{cor}\label{cor:}
    Let $n \geq 5$. There is no subgraph of $\C(\Sigma_n)$ that is isomorphic to $\KG'(n)$. 
\end{cor}

\begin{proof}
    Assume for the purposes of contradiction that there is an injective graph homomorphism $\iota\colon \KG'(n) \to \C(\Sigma_n)$. Then $\Phi \circ \iota$ is an endomorphism of $\KG'(n)$. Since $\KG'(n)$ is a core,  $\Phi \circ \iota$ must be an automorphism. We can moreover assume that $\Phi \circ \iota$ is the identity without loss of generality. Thus, if $x \in \KG(n,k)$, then $\iota(x) \in \Phi^{-1}(\KG(n,k))$ for all $k \in \{1,\dots, \lfloor n/2 \rfloor \}$. In particular, if $A|A^c$ is a vertex of $\KG(n,2)$ with $|A| = 2$, then $\iota(A|A^c)$ is a curve on $\Sigma_n$ that splits the punctures labeled with the elements of $A$ from those labeled by $A^c$. We can therefore draw a copy of the complete graph $K_n$ on a sphere by connecting pairs of the punctured points by an arc contained in disks bounded by these curves. This would be a drawing of $K_n$ where the edge between any two disjoint pairs of vertices does not intersect (because the curves used to draw the arcs are disjoint). By the Hanani--Tutte theorem, this would contradict the non-planarity of $K_n$ for $n\geq 5$ \cite{Hanani_Planar,Tutte_planar}. Hence, there is no injective homomorphism $\KG'(n) \to \C(\Sigma_n).$
\end{proof}

\bibliographystyle{amsalpha}

\bibliography{kneser_bib}

\end{document}